\documentclass[leqno,a4paper]{article}
\usepackage{amsmath,
amsthm,amssymb}
\usepackage{eufrak}

\newcommand{\mT}{\mathcal{T}}

\newcommand{\R}{{\mathbb R}}

\newcommand{\oO}{{\overline{\Omega}}}

\newcommand{\ut}{{\underline{t}}}

\newcommand{\pa}{{\partial}}

\newcommand{\Om}{{\Omega}}
\newcommand{\Ga}{{\Gamma}}
\newcommand{\Si}{{\Sigma}}
\newcommand{\La}{{\Lambda}}

\newcommand{\ka}{{\kappa}}

\newcommand{\bt}{{\beta}}

\newcommand{\ua}{{\underline{a}}}
\newcommand{\ub}{{\underline{b}}}

\newcommand{\n}{{\underline{n}}}
\newcommand{\ui}{{\underline{i}}}
\newcommand{\ou}{{\underline{u}}}
\newcommand{\ov}{{\underline{v}}}
\newcommand{\oo}{{\underline{\omega}}}
\newcommand{\utau}{{\underline{\tau}}}

\newcommand{\be}{\begin{equation}}
\newcommand{\ee}{\end{equation}}

\def\vs1{\vspace{1ex}}
\newtheorem{definition}{Definition}[section]
\newtheorem{theorem}{Theorem}[section]
\newtheorem{lemma}[theorem]{Lemma}
\newtheorem{proposition}{Proposition}[section]
\newtheorem{corollary}{Corollary}[section]
\newtheorem{remark}{Remark}[section]

\numberwithin{equation}{section}
\begin{document}
\title{\bf\normalsize  A MISSED PERSISTENCE PROPERTY FOR THE EULER EQUATIONS,
AND ITS EFFECT ON INVISCID LIMITS }
\author{by H.~Beir\~ao da Veiga and F. Crispo}
\maketitle \textit{\phantom{aaaaaaaaaaaaaaaaaaaaaaaaaa}}

\begin{abstract}
We consider the problem of the \emph{strong} convergence, as the
viscosity goes to zero, of the solutions to the three-dimensional
evolutionary Navier-Stokes equations under a Navier slip-type
boundary condition to the solution of the Euler equations under the
zero-flux boundary condition. In spite of the arbitrarily strong
convergence results proved in the flat boundary case, see
\cite{bvcrgr}, it was shown in reference \cite{bvcr-contra} that the
result is false in general, by constructing an explicit family of
smooth initial data in the sphere, for which the result fails. Our
aim here is to present a more general, simpler and incisive proof.
In particular, counterexamples can be displayed in arbitrary,
smooth, domains. As in \cite{bvcr-contra}, the proof is reduced to
the lack of a suitable persistence property for the Euler equations.
This negative result is proved by a completely different approach.

\vspace{.2cm}

{\bf Mathematics Subject Classification} 35Q30, 76D05, 76D09.

\vspace{.2cm}

{\bf Keywords.} Euler and Navier-Stokes equations, inviscid limit,
slip boundary conditions.

\end{abstract}

\bibliographystyle{amsplain}

\section{Introduction and some results.}
We investigate \emph{strong convergence, up to the boundary}, as
$\nu \rightarrow \,0\,,$ of the solutions $\,\ou^\nu\,$ of the
Navier-Stokes equations
\begin{equation}
\left\{
\begin{array}{l}
\partial_t\,\ou^\nu +\,(\ou^\nu \cdot\,\nabla)\,\ou^\nu+\,\nabla\,p=\,\nu\,\Delta\,\ou^\nu\, ,\\
{div}\,\ou^\nu =\,0\,,\\
\ou^\nu(0)=\,\ua\,,
\end{array}
\right. \label{1;4}
\end{equation}
under the boundary condition
\begin{equation}
\left\{
\begin{array}{l}
\ou^\nu\cdot \n=\,0,\\
\oo^\nu \times \,\n=\,0\,,
\end{array}
\right. \label{bcns}
\end{equation}
to the solution $\ou$ of the Euler equations
\begin{equation}
\left\{
\begin{array}{l}
\partial_t\,\ou +\,(\ou\cdot\,\nabla)\,\ou+\,\nabla\,p=\,0,\\
{div}\,\ou=\,0\,,\\
\ou(0)=\,\ua\,,
\end{array}
\right. \label{1;4e}
\end{equation}
under the zero-flux boundary condition
\begin{equation}
\ou\cdot\,\n=\,0\,.%
\label{bceu}
\end{equation}
We are interested in the three-dimensional situation. Here, and in
the sequel, we use the notation $\,\oo=\, curl \,\ou\,$.\par%
\begin{definition}
We say that a vector field $\,\ua\,$ is \emph{admissible} if it is
smooth and divergence free in  $\,\oO\,$, and satisfies the boundary
condition
\begin{equation}
\left\{
\begin{array}{l}
\ua\cdot \n=\,0,\\
\ub \times \,\n=\,0\,,
\end{array}
\right. \label{bcnes}
\end{equation}
where
\begin{equation}
\ub=\,curl \,\ua\,.%
\label{bes}
\end{equation}
\end{definition}
The above condition is usually referred to as a Navier slip-type
boundary condition.\par%
Recently the vanishing viscosity limit
problem under the above or similar Navier type conditions has been
studied in \cite{bvcr}, \cite{clopeau}, \cite{Kell}, \cite{Lopes},
in the 2D case, and in \cite{bvcr}, \cite{bvcr2},
\cite{bvcr-contra}, \cite{Chen}, \cite{xin}, in the 3D case. See
also \cite{XXW} for the magnetohydrodynamic system and
\cite{Belloutet} for a different approach to the inviscid limit for
the slip-type boundary value problem.

\vspace{0.2cm}

\emph{The domain $\Om\,:\,$} In the sequel $\,\Om$ is a bounded open
set in $\R^3\,$ locally situated on one side of its boundary, a
smooth manifold $\Ga$. The boundary $\,\Ga\,$ may consist of a
finite number of disjoint, connected components, $\,\Ga_j\,$,
$\,j=\,0,\,1,\,...,m\,,$ $\,m \geq\,0\,$. $\,\Ga_0$ denotes the
``external boundary". If $\,\Ga\,$ is not simply-connected we assume
the typical existence of $\,N\,$ mutually disjoint and transversal
cuts, after which $\,\Om\,$ becomes simply-connected
(see \cite{foias-temam} and \cite{temamlibro} for details).%

\vspace{0.2cm}

We denote by $\n=\,(n_1,\,n_2,\,n_3)\,$ the unit outward normal to
$\Ga\,,$ and denote by $\,\ka_j(s)\,, j=\,1,\,2\,,$ the
\emph{principal curvatures} of $\,\Ga\,$ at a point $\,s\,$. We set
\begin{equation}%
\Si=\,\{s\in\,\Ga\,:\, \ka_j(s)\neq\,0\,, \quad
j=\,1,\,2\,\}\,.%
\label{xigma}
\end{equation}
$\,\Si\,$ is the subset of boundary-points where the Gaussian
curvature $\,\ka_1\,\ka_2\,$ does not vanish. It is worth noting
that, for $\,\Om\,$ as above, $\,\Si\,$ is never empty.
Mostly, $\,\Si\,$ coincides with $\Ga\,$ itself.%

\vspace{0.2cm}

We recall that application of the operator $\,curl\,$ to the first
equation \eqref{1;4e} leads to the well known Euler vorticity
equation
\begin{equation}
\pa_t\,\oo\,-\,curl \,(\ou \times\,\oo)\, =\,0\,.%
\label{boh2}
\end{equation}
\begin{definition}
\rm{ We say that the Euler equations \eqref{1;4e}, under the
boundary condition \eqref{bceu}, satisfy the \emph{persistence
property} (with respect to the boundary condition $\oo \times
\,\n=\,0\,,$ and to the initial data $\ua\,$), if $\oo(0) \times
\,\n=\,0\,$ on $\,\Ga\,$ implies the existence of some $\,t_0>\,0\,$
(which may depend on $\,\ua\,$) such that $\oo(t) \times \,\n=\,0\,$
on $\,\Ga\,$, for each $\,t \in\,(0,\,t_0)\,$. Furthermore, we say
that the persistence
property holds, if it holds for all (smooth) initial data $\,\ua\,.$}%
\label{persistes}
\end{definition}

\vspace{0.2cm}

\begin{definition}
\rm{By \emph{strong convergence} we mean any (sufficiently strong)
convergence in $\,(0,\,T) \times \,\Om\,$  such that if
$\,\ou^\nu\,$  converges to $\,\ou\,$ with respect to this
convergence, and if $\oo^\nu \times \,\n=\,0\,$ on $\,\Ga\,$, then
necessarily $\oo \times \,\n=\,0\,$ on $\,(0,\,T) \times \,\Ga\,$.
\emph{Strong inviscid limit} is defined accordingly.}%
\label{hetlixas}
\end{definition}
Examples of strong convergence are, for instance, convergence in
$L^1(0,\,T; W^{s,q})\,$, for some $q >\, 1\,,$ and some
$s>\,1+\,\frac1q \,$, and convergence in $L^1(0,\,T; W^{2,1})\,.$
Recall that $\,L^1(0,\,T)\,$ convergence implies a.e. convergence in
$\,(0,\,T)\,,$ for suitable ``sub-sequences".%

\vspace{0.2cm}

A strong inviscid limit result, without a spatial boundary, was
proved in \cite{K3}. See also the more recent papers \cite{bvkaz}
and \cite{Masm}. In \cite{xin}, \cite{bvcr}, \cite{bvcr2}, and
\cite{bvcrgr}, strong inviscid limit results are proved under a
flat-boundary assumption. However, in the case of non-flat
boundaries the problem remained open. The arbitrarily strong
convergence results proved in \cite{bvcrgr}, some estimates proved
for non-flat boundaries in \cite{bvcr} and \cite{bvcr2}, and the
strong convergence results available in the two-dimensional case,
led to the conviction that strong convergence results held in the
general three-dimensional case, at least in ``moderately strong"
topologies. In spite of this guess, in reference \cite{bvcr-contra}
we have shown that the result is false in a sphere.

\vspace{0.2cm}

In reference \cite{bvcr-contra}, section 2, the following result is
proved.
\begin{theorem}
Let the initial data $\,\ua\,$ be admissible. Then:\par%
a) If a strong inviscid limit result holds, then necessarily the
Euler equations \eqref{1;4e}, \eqref{bceu} enjoy the persistence property.\par%
b) If the persistence property holds, then necessarily
\begin{equation}
curl\,(\ua \times \,\ub)\,\times \,\n=\,0%
\label{vanas}
\end{equation}
everywhere on $\,\Ga\,$.
\label{doutro}
\end{theorem}
The proof of this particularly useful result is astonishingly
simple. Actually, a strong inviscid limit result immediately implies
the persistence property for the Euler equations. Assume now the
persistence property. External multiplication of \eqref{boh2} by
$\,\n\,$ gives
\begin{equation}
\pa_t\,(\oo \times \, \n \,)-\,curl\,(\ou \times
\oo)\,\times \n =\,0\,.%
\label{vort}
\end{equation}
Since the persistency property  holds, the time derivative in the
above equation must vanish on $\,\Gamma\,,$ at time $\,t=\,0\,$. So,
the second term must verify this same property. That such a simple
short-cut remained hidden may be due to its being extremely elementary.\par%
For convenience, we state the above result in the following
equivalent form.
\begin{theorem}
Let $\,\ua\,$ be an admissible vector field. Then:\par%
a) If, in some point $\,x_0\in\,\Ga\,,$ the inequality
\begin{equation}
curl\,(\ua \times \,\ub)\,\times \,\n \neq\,0%
\label{noponto}
\end{equation}
holds, then the persistence property, with respect to the initial
data $\ua\,,$ fails.\par%
b) If the persistence property fails then, necessarily, any strong
inviscid limit result fails.
\label{problas}
\end{theorem}
Clearly, if the condition \eqref{noponto} holds in $\,x_0\,,$ it
holds in a $\,\Ga-$neighborhood of this same point.\par%
It follows from the above theorem that in order to prove the failure
of the persistence property and, a fortiori, that of strong inviscid
limit results, it is sufficient to show the existence of admissible
vector fields $\,\ua\,$  which satisfy \eqref{noponto} somewhere in
$\,\Ga\,$. We will show that, given $\Om\,$ as above, there exist a
large family of such $\,\ua\,$. For fixing ideas, we state our main
result in its simplest form. Actually, our argument leads to more
precise and deeper versions of Theorem \ref{teobanilias} below, as
the interested reader may verify. However, to be clear and concise,
we limit ourselves to the following statement.
\begin{theorem}
Let $\,x_0\,$  be a boundary point where the Gaussian curvature does
not vanish, that is
\begin{equation}
\ka_1(x_0)\,\ka_2(x_0) \neq \,0\,.
\label{gauss}
\end{equation}
Then, there are admissible vector fields $\,\ua\,$ for which the
inequality \eqref{noponto} holds at a sequence of boundary points
$\,x_n\,,$ convergent to $\,x_0\,$. So, persistence property and
strong vanishing limit results fail in general.%
\label{teobanilias}
\end{theorem}
By the way, note that when \eqref{noponto} or \eqref{gauss} hold in
some point, they hold in a neighborhood of this same point.
\section{Remarks}
Boundedness of $\,\Om\,$ is not essential here. The existence of
points $\,x_0 \in\,\Ga\,$ where the Gaussian curvature does not
vanish is sufficient to apply our argument.%

\vspace{0.2cm}

We overlook to consider boundary points where only one of the two
principal curvatures does not vanish. It could be of some interest
to study this case, by taking into account equation \eqref{rotsen}
below. This situation applies, for instance, to the case in which
$\,\Om\,$ is a cylinder, and slip boundary conditions are prescribed
only on the lateral boundary (a developable surface), and
periodicity is assumed in the ruling's direction.%

\vspace{0.2cm}

We prove the identity \eqref{rotsen} in the case of non umbilical
points $\,x_0\,$. If the principal curvatures coincide at $\,x_0\,,$
our proof needs some modification since the local system of
coordinates used in the proof may not exist (for instance, it does
not exist if $\,\Ga\,$ is spherical near $\,x_0\,$). We leave to the
interested reader the proof of \eqref{rotsen} under the assumption
$\,\ka_1=\,\ka_2\,$. Note, however, that this particular situation
can be bypassed here since if in a connected, closed, smooth surface
all points are umbilical, the surface is a sphere. And this case was
already treated in reference \cite{bvcr2}.

\vspace{0.2cm}

We note that our negative result does not exclude inviscid limit
results in weaker spatial norms, such as $H^1(\Omega)$. In this
regard we refer to paper \cite{xin2}, where this kind of convergence
is proved under one of the following additional assumptions :
$\,(\,\oo^\nu-\oo\,)\times\, \n\,=0$ on $(0,T)\times\Ga\,$, where
$\,\oo^\nu\,$ and $\,\oo$ are the vorticity for the Navier-Stokes
and Euler equations respectively, or $\ub=0$ on $\Ga$\,.

\begin{remark}
\rm{On flat portions of the boundary, the slip boundary condition
\eqref{bcns} coincides with the classical Navier boundary condition
\begin{equation}
\left\{
\begin{array}{l}
\ou\cdot \n=\,0,\\
\ut \cdot \,\utau=\,0\,,
\end{array}
\right. \label{boundary-ho}
\end{equation}
where $\,\utau\,$ stands for any arbitrary unit tangential vector.
Here $\,\ut\,$ is the stress vector defined by $\,\ut
=\,\mT\cdot\,\n\,$, where the stress tensor $\,\mT\,$ is defined by
\begin{equation*}
\mT =-\,p\,I+\,\frac{\nu}{2} \,(\nabla \ou+\nabla \ou^T)\,.
\end{equation*}
These conditions were introduced by Navier  in \cite{Navier}, and
derived by Maxwell in \cite{max} from the kinetic theory of gases.
For general boundaries
\begin{equation}
\ut \cdot \,\utau=\,\frac{\nu}{2} \,(\oo \times \,\n) \cdot \,\utau
\,-\,\nu \,\mathcal{K}_{\tau}\, \ou \cdot \,\utau\,,%
\label{parec}
\end{equation}
where $\,\mathcal{K}_{\tau}\,$ is the principal curvature in the
$\,\utau\,$ direction, positive if the corresponding center of
curvature lies inside $\,\Om\,$.\par%
Note that our counter-example in \cite{bvcr-contra} and the results
presented here do not exclude that strong inviscid limit results
 hold under the Navier boundary condition \eqref{boundary-ho}
in the non-flat boundary case. To prove or disprove this kind of
result remains a challenging open problem.}%
\label{tolosa}
\end{remark}
\section{A main identity}
The Theorem \ref{problas} places the non-linear term $\,curl\,(\ua
\times \,\ub)\,\times \,\n\,$ in a central position. So, reducing
the order of this second order term is here very helpful. This
reduction is done by proving the identity \eqref{rotsen}. In our
context, this identity has another valuable merit. It makes explicit
a precise dependence on curvature,
which is essential in the sequel.%

\vspace{0.2cm}

Before stating Theorem \ref{geom} we introduce some
notation.\par%
Given $\,x_0 \in \,\Ga_0\,$ we introduce, in a sufficiently small
neighborhood of $\,x_0\,$, a suitable system of orthogonal
curvilinear coordinates $\,(\xi_1,\,\xi_2,\,\xi_3)\,$. See, for
instance, \cite{krey}, in particular, chapter 8, paragraph 89. The
surface $\Ga$ is locally described by the equation $\xi_3=\,0$,
moreover the surfaces $\xi_3=\,constant$ are parallel to $\Ga$ in
the usual sense. They are located at distance $\,|\,\xi_3\,|\,$ from
$\,\Ga\,\,$. The coordinate $\xi_3$ increases outside $\Om$.
Further, on each parallel surface the lines $\xi_j=\,constant\,$,
$\,j=\,1,\,2\,$ are lines of curvature, hence tangent to a principal
direction. Recall that the lines of curvature on parallel surfaces
correspond to each other. The point $\,x_0\,$ has zero coordinates.
We denote by $\ui_{j}$ the unit vector, tangent to the $\xi_j$ line,
and pointing in the direction of increasing $\xi_j$. Hence, at each
point of a parallel surface, the vectors $\ui_{j}$, $j=\,1,\,2$, are
tangent to the principal directions. The corresponding normal
curvatures, the so called principal curvatures, are denoted by
$\,\ka_j$, $j=\,1,\,2$. They take the maximum and the minimum of the
set of all the normal curvatures. The unit vector $\ui_3$ coincides
with the normal $\,\n\,$. Roughly speaking, concerning signs, it is
sufficient to remark that the curvature of a normal section of $\Ga$
at a point $\,x\,$ is positive whenever the normal section of $\Om$
is convex at $\,x\,$.\par%
Components of vector fields are with respect to the orthogonal basis
$\ui_j$, $j=\,1,\,2\,,3\,$. For instance
$\,\ua=\,a_1\,\ui_1+\,a_2\,\ui_2+\,a_3\,\ui_3\,.$\par%
A point is \emph{umbilical} if $\,\ka_1=\,\ka_2\,$. If
$\,\ka_1=\,\ka_2=\,0\,$ the point is a \emph{planar} (or parabolic
umbilical) point. As already remarked we assume, for convenience,
that $\,x_0\,$ is not umbilical. In this case some modifications are needed.\par%
At each point, the ordered orthogonal basis
$\ui_{1},\,\ui_{2},\,\ui_{3}$ is assumed to be positively
(right-handed) oriented. If $\,s(\xi_j)$ denotes the arc length
along a $\xi_j$-line, the (positive) $h_j\,$ scale functions are
defined by
$$
h_j=\,\frac{d\,s(\xi_j)}{d\,\xi_j}\,.
$$
Note that $\,h_3=\,1\,$ everywhere. In particular,
\begin{equation}
\frac{\pa \,h_3}{\pa\,\xi_j}=\,0\,, \quad j=\,1,\,2\,.%
\label{wabe}
\end{equation}
We recall that
\begin{equation}
\ka_1=\,\frac{1}{h_3 \,h_1} \,\frac{\pa \,h_1}{\pa\,\xi_3}\,,\quad
\ka_2=\,\frac{1}{h_3 \,h_2} \,\frac{\pa \,h_2}{\pa\,\xi_3}\,.
\label{kapa}
\end{equation}

\begin{theorem}
Let $\,\ua\,$ be an admissible vector field. Then, the identity
\begin{equation}
curl\,(\ua \times \,\ub)\,\times \,\n=\,-\,2\,b_3
\,(\ka_2\,a_2\,\ui_1 -\,\ka_1\,a_1\,\ui_2\,)%
\label{rotsen}
\end{equation}
holds on $\,\Ga\,.$%
\label{geom}
\end{theorem}
\begin{proof}
We recall the following expression for the $curl$ of a vector field
$\ov$ in curvilinear, orthogonal, coordinates:
\begin{equation}
\begin{array}{l}
curl \,\ov=\,\frac{1}{h_2\,h_3} \,\big[\,\frac{\partial \,(h_3
\,v_3)}{\partial\,\xi_2}-\, \frac{\partial \,(h_2
\,v_2)}{\partial\,\xi_3} \, \big]\,\ui_1+\\
\\
\frac{1}{h_3\,h_1} \,\big[\,\frac{\partial \,(h_1
\,v_1)}{\partial\,\xi_3}-\, \frac{\partial \,(h_3
\,v_3)}{\partial\,\xi_1} \, \big]\,\ui_2 + \,\frac{1}{h_1\,h_2}
\,\big[\,\frac{\partial \,(h_2 \,v_2)}{\partial\,\xi_1}-\,
\frac{\partial \,(h_1
\,v_1)}{\partial\,\xi_2} \, \big]\,\ui_3 \,.%
\end{array}
\label{exprot}
\end{equation}
Since (recall that $\,\ub=\, curl \,\ua\,$)
\begin{equation}
curl\,(\ua \times \,\ub)\,\times \,\n=\,[curl\,(\ua \times \,\ub)]_2
\,\ui_1 -\,[curl\,(\ua \times \,\ub)]_1 \,\ui_2\,,%
\label{semn}
\end{equation}
we are interested in the two tangential components of $curl\,(\ua
\times \,\ub)$. Hence, by setting $\ov=\,\ua \times \,\ub\,$ in
\eqref{exprot}, we want to determine the two first terms on the left
hand side of \eqref{exprot}. Due to the similarity of these two
terms it is sufficient to treat one of them. We consider the first
one. This leads to
\begin{equation}
[curl\,(\ua \times \,\ub)]_1=\,\frac{1}{h_2\,h_3}\, \frac{\partial
\,}{\partial\,\xi_2}\,[\,h_3\,(\ua_1\,\ub_2-\,\ua_2\,\ub_1\,)\,]-\,\frac{1}{h_2\,h_3}\,
\frac{\partial\,}{\partial\,\xi_3}\,[\,h_2\,(\ua_3\,\ub_1-\,\ua_1\,\ub_3\,)\,]\,.%
\label{ium}
\end{equation}
Note that
$$
a_3=\,b_1=\,b_2=\,0
$$
for $\,\xi_3=\,0\,$, hence
$$
\frac{\partial \,a_3}{\partial\,\xi_j}=\,\frac{\partial
\,b_i}{\partial\,\xi_j}=\,0\,,
$$
for $\,i,\,j=\,1,\,2\,$.\, It follows that the first term on the
right hand side of equation \eqref{ium} and the "first half" of the
second term vanish on $\Ga$. So,
$$
[curl\,(\ua \times \,\ub)]_1= \,\frac{1}{h_2\,h_3}\,\frac{\partial
\,}{\partial\,\xi_3}\,(\,h_2\,a_1\,b_3\,).
$$
Consequently,
\begin{equation}
[curl\,(\ua \times \,\ub)]_1=\,\ka_2\,a_1\,b_3
+\,\frac{1}{h_3}\,b_3\,\frac{\partial\,a_1}{\partial\,\xi_3}+\,
\frac{1}{h_3}\,a_1\,\frac{\partial\,b_3}{\partial\,\xi_3}\,.%
\label{ein}
\end{equation}
Since $\,b_2 =\,0\,$ on $\,\Ga\,$, it follows from \eqref{exprot}
applied to $\ua$ that
\begin{equation}
\frac{\partial \, a_1}{\partial\,\xi_3}=\,-
\frac{a_1}{h_1}\,\frac{\partial \, h_1}{\partial\,\xi_3}\,.%
\label{zwei}
\end{equation}
We have appealed to $a_3=\,\frac{\partial \,
a_3}{\partial\,\xi_1}=\,0\,$ on $\Ga$, and to \eqref{wabe}.\par%
Next
\begin{equation}
div \,\ub=\, \frac{1}{h_1\,h_2\,h_3} \,\Big\{ \,\frac{\partial
\,(h_2\,h_3\,b_1)}{\partial\,\xi_1}+\, \frac{\partial
\,(h_3\,h_1\,b_2)}{\partial\,\xi_2}+\, \frac{\partial
\,(h_1\,h_2\,b_3)}{\partial\,\xi_3}\,\Big\}=\,0\,,%
\label{mah}
\end{equation}
in $\,\Om\,$. From $b_i=\,\frac{\partial
\,b_i}{\partial\,\xi_j}=\,0\,$ on $\Ga\,$, $\,i,\,j=\,1,\,2\,$, one
gets
$$
 \frac{1}{h_1\,h_2\,h_3} \, \frac{\partial
\,(h_1\,h_2\,b_3)}{\partial\,\xi_3}=\,0\,,%
$$
on $\Ga\,.$ This leads to
\begin{equation}
\frac{1}{h_3}\, \frac{\partial
\,b_3}{\partial\,\xi_3}=\,-\,(\,\ka_1+\,\ka_2\,)\,b_3\,.%
\label{drei}
\end{equation}
From \eqref{ein}, \eqref{zwei}, and \eqref{drei}, it readily follows
that
\begin{equation}
[curl\,(\ua \times \,\ub)]_1=\,-\,2\,\ka_1\,a_1\,b_3\,.%
\label{vier}
\end{equation}
Analogously we may prove that
\begin{equation}
[curl\,(\ua \times \,\ub)]_2=\,-\,2\,\ka_2\,a_2\,b_3\,.%
\label{funf}
\end{equation}
\end{proof}
Note that the above result has a local character. In fact the proof
immediately applies to show the following. Let $\,U\,$ be an
arbitrary open set such that $\,\Ga_0=\,U\,\cap\,\Ga\,$ is not
empty. Further, assume that $\,\ua\,$ is a smooth divergence free
vector field in $\, \overline{\Om}_0\,,$ where
$\,\Om_0=\,U\,\cap\,\Om\,,$ which satisfies the boundary conditions
\eqref{bcnes} on $\,\Ga_0\,.$ Then, the identity \eqref{rotsen}
holds on $\,\Ga_0\,.$\par%
The next result follows from the theorems \ref{problas} and
\ref{geom}.
\begin{corollary}
Let $\,\ua\,$ be admissible. If, at some point $\,x_0 \in \,\Ga\,,$
\begin{equation}
b_3 \neq \,0\,,%
\label{cone}
\end{equation}
and if (at least) one of the two following conditions
\begin{equation}
\ka_1\,a_1 \neq\,0 \quad \textrm{or} \quad \ka_2\,a_2 \neq \,0\,,%
\label{cond}
\end{equation}
hold, then the inequality \eqref{noponto} takes place. In
particular, the persistence property fails. Consequently,
any strong inviscid limit result is false.%
\label{theocorol}
\end{corollary}
Corollary \ref{theocorol} leads us to look for points $\,x_0 \in
\,\Ga\,$ for which \eqref{cone} and \eqref{cond} hold
simultaneously. Hence, to show that these two inequalities are not
independent is here of great help. The following result holds.
\begin{proposition}
Let $\,\ua\,$ be admissible, and assume that \eqref{cone} holds in
some point $\,x_0\in\,\Ga\,$. Then, there is a sequence of boundary
points $\,x_n\in\,\Ga\,,$ convergent to $\,x_0\,$, and such that
\begin{equation}
a_j(x_n) \neq \,0\,,%
\label{conedi}
\end{equation}
for, at least, one of the two index $\,j\,$. So, if $\,x_0 \in
\,\Si\,$,
\begin{equation}
b_3\,\ka_j\,a_j \neq \,0\,,%
\label{bka}
\end{equation}
at $\,x_n\,,$ at least for sufficient large values of $\,n\,$.%
\label{propco}
\end{proposition}
\begin{proof}
If there are neighborhoods $\,U_j\,$, $j=\,1,\,2\,$, where
\eqref{conedi} does not hold, then $\,a_1=\,a_2=\,0\,$ in $\,U_1
\cap\,U_2\,.$ On the other hand, equation \eqref{exprot} shows that
\begin{equation}
b_3=\,\frac{1}{h_1\,h_2} \,\Big(\,\frac{\partial \,(h_2
\,a_2)}{\partial\,\xi_1}-\, \frac{\partial \,(h_1
\,a_1)}{\partial\,\xi_2} \, \Big) \,.%
\label{pincas}
\end{equation}
It follows that $\,b_3=\,0\,$ on $\,U_1 \cap\,U_2\,.$ This
contradicts \eqref{cone}.
\end{proof}
From proposition \ref{propco} it follows that to prove the theorem
\ref{teobanilias} it is sufficient to show that, given any $\,x_0
\in \,\Si\,$, there is an admissible vector field $\,\ua\,$ such
that $\,b_3 \equiv \,b \times \,\n \neq\,0\,$ at $\,x_0\,$. This is
the aim of the next section.
\section{Proof of Theorem \ref{teobanilias}}
In this section we show an elementary way to explicitly construct
admissible vector fields $\,\ua\,$ for which \eqref{cone} holds at
any fixed $\,x_0\,\in\,\Ga\,.$ By choosing  $\,x_0\,\in\,\Si\,,$ we
prove the theorem \ref{teobanilias}. In the following, topological
properties of subsets of the boundary $\Ga$ concern the $\Ga$
topology (and not the $\oO\,$ topology).

\vspace{0.2cm}

Let $\,\bt(s)\,$ be a smooth real function on $\,\Ga\,$ such that
\begin{equation}
\int_{\Ga_j} \bt(s) \,ds=\,0\,,\quad \,j=\,0,...,\,m\,,%
\label{best}
\end{equation}
and define
\begin{equation}
\ub(s)=\,b(s)\,\n\,.%
\label{ubes}
\end{equation}
Clearly,
\begin{equation}
\int_{\Ga_j} \ub(s)\cdot\,\n \,ds=\,0\,\quad ,\,j=\,0,...,\,m\,.%
\label{bezero}
\end{equation}
It is well known that, under assumption \eqref{bezero}, there exists
in $\,\oO\,$ an extension $\,\ub(x)\,$ of $\,\ub(s)\,$ such that
\begin{equation}
div\,\ub(x)=\,0\,.%
\label{divo}
\end{equation}
On the other hand, it is well known that, under the constraints
\eqref{bezero} and \eqref{divo}, the linear problem
\begin{equation}
\left\{
\begin{array}{l}
div\,\ua =\,0\,,\\
curl\,\ua=\,\ub\,, \quad \textrm{in} \quad \Om\,,\\
\ua \cdot\,\n=\,0\,, \quad \textrm{on} \quad \Ga
\end{array}
\right.%
\label{vaquideo}
\end{equation}
is always solvable. This existence result is sufficient to our
purposes. However we recall that the solution is unique if $\,\Om\,$
is simply connected and that, in general, the kernel (set
$\,\ub=\,0\,$) of the related linear map has dimension $\,N\,$. For
a very clear and complete treatment of this, and related, problems
we refer the reader to the section 1, in reference \cite{foias-temam}.\par%
The following result shows a crucial advantage of our approach.
\begin{lemma}
For each $\,\bt(s)\,$ as above, the vector field $\,\ua\,$ is
admissible.
\end{lemma}
In fact, \eqref{ubes} together with the second equation
\eqref{vaquideo}, implies the second boundary condition
\eqref{bcnes}.

\vspace{0.2cm}

Given $\,\bt(s)\,$ as above we denote by $\,\La[\bt]\,$ the set of
boundary points defined by
\begin{equation}%
\La[\bt]=\,\{s\in\,\Ga\,:\, \bt(s)\neq\,0\,\}\,.%
\label{labas}
\end{equation}
The following result is obvious.
\begin{lemma}
We may choose functions $\,\bt\,$ for which $\,\La[\bt]=\,\Ga\,,$
except for $\,m+\,1\,$ arbitrary, closed simple curves, $\,C_j
\subset \,\Ga_j\,$, $j=\,0,\,...,\,m\,$. These curves may be
arbitrarily chosen.%
\label{meios}
\end{lemma}
\emph{Proof of Theorem \ref{teobanilias}}: Let $\,x_0\,$ be an
arbitrary, but fixed, boundary-point where the Gaussian curvature
does not vanish. Taking into account Lemma \ref{meios}, we fix a
real function $\,\bt(s)\,$ such that $\,\bt(x_0) \neq\,0\,,$ and
construct $\,\ua\,$ as above. From proposition \ref{propco}, and
theorems \ref{geom} and \ref{problas}, the thesis follows.
\begin{remark}
Define (recall \eqref{noponto})
$$
K[\bt]=\,\{s\in\,\Ga\,:\, curl\,(\ua \times \,\ub)\,\times \,\n \neq\,0\,\}\,.%
$$
By appealing to our argument, we may prove that the set $\,K[\bt]\,$
is dense in $\,\Si \cap\,\La[\bt]\,,$ and that there are functions
$\,\bt(s)\,$ for which $\,K[\bt]\,$ is dense in
$\Si\,$.%
\end{remark}

{\it Acknowledgments.}\; The work of the second author was supported
by INdAM (Istituto Nazionale di Alta Matematica) through a Post-Doc
Research Fellowship at Dipartimento di Matematica Applicata,
University of Pisa.

\end{document}